\documentclass[preprint,12pt]{elsarticle}

\usepackage[utf8]{inputenc}
\usepackage{amsmath}
\usepackage{amsfonts}
\usepackage{amssymb}
\usepackage{amsthm}
\usepackage{enumitem} 
\usepackage{cite} 
\usepackage{bm}
\usepackage{xcolor}
\usepackage{fullpage}
\usepackage{comment}
\usepackage{pgfplots} 
\usepackage{tikz} 
\usetikzlibrary{positioning}
\usepgfplotslibrary{fillbetween} 
\definecolor{green_color_blind}{RGB}{102,194,165}
\definecolor{orange_color_blind}{RGB}{252,141,98}
\definecolor{blue_color_blind}{RGB}{141,160,203} 
\usepackage{algorithm}  
\usepackage{algpseudocode} 
\usepackage{url} 
\usepackage{cite} 

\usepackage[acronym]{glossaries} 
\newacronym{LMI}{LMI}{Linear Matrix Inequality}
\newacronym{QDF}{QDF}{Quadratic Difference Form}
\newacronym{LTI}{LTI}{Linear Time-Invariant}
\newacronym{ADMM}{ADMM}{Alternating Direction Method of Multipliers}
\newacronym{LPV}{LPV}{Linear Parameter-Varying}

\newcommand{\R}{\mathbb{R}}
\renewcommand{\S}{\mathbb{S}}
\DeclareMathOperator{\im}{im}
\DeclareMathOperator{\rank}{rank}
\newcommand{\lmin}{\ell_{\mathrm{min}}}
\newcommand{\nmin}{n_{\mathrm{min}}}
\newcommand{\true}{\ast}
\newcommand{\B}{\mathfrak{B}}
\newcommand{\Zp}{{\mathbb{Z}_+}}

\newcommand{\calH}{\mathcal{H}}

\newcommand{\Lmp}{\mathfrak{L}_{m,p}}
\newcommand{\Lmpcont}{\mathfrak{L}_{m,p}^{\mathrm{cont}}}
\newcommand{\LmpLcont}{\mathfrak{L}_{m,p,L}^{\mathrm{cont}}}

\makeatletter
\renewcommand*\env@matrix[1][*\c@MaxMatrixCols c]{%
  \hskip -\arraycolsep
  \let\@ifnextchar\new@ifnextchar
  \array{#1}}
\makeatother

\newcommand {\T}{\mathbb{T}} 	
\newcommand {\W}{\mathbb{W}} 	

\newtheorem{theorem}{Theorem}
\newtheorem{lemma}[theorem]{Lemma}

\newtheorem{definition}[theorem]{Definition}

\newtheorem{remark}[theorem]{Remark}

\newcommand{\rB}[2]{\B\!\!\mid_{[{#1},{#2}]}}
\newcommand{\rBtrue}[2]{\B_\true\!\!\mid_{[{#1},{#2}]}}

\definecolor{greenpigment}{rgb}{0.0, 0.65, 0.31}

\journal{Systems and Control Letters}

\begin{document}

\begin{frontmatter}

\title{From time series to dissipativity of linear systems \\ with dynamic supply rates}

\author[1]{H.J. van Waarde}\fnref{funding}
\affiliation[1]{organization={Bernoulli Institute for Mathematics, Computer Science and Artificial Intelligence, University of Groningen},
            country={The Netherlands},
            email={ (h.j.van.waarde@rug.nl)}}
\fntext[funding]{Henk van Waarde acknowledges financial support by the Dutch Research Council (NWO) under the Talent Programme Veni Agreement (VI.Veni.22.335). A. Padoan acknowledges the support of the Natural Sciences and Engineering Research Council of Canada (NSERC). Grant numbers: RGPIN-2025-06895 and DGECR-2025-00382.}
\author[2]{J. Coulson}
\affiliation[2]{organization={Department of Electrical and Computer Engineering, University
of Wisconsin-Madison},
            country={USA},
            email={ (jeremy.coulson@wisc.edu)}}
\author[3]{A. Padoan} 
\affiliation[3]{organization={Department of Electrical and Computer Engineering, University of British Columbia},
            country={Canada},
            email={ (alberto.padoan@ubc.ca)}}

\begin{abstract}
This paper studies the problem of verifying dissipativity of linear time-invariant (LTI) systems using input-output data. We leverage behavioral systems theory to express dissipativity in terms of quadratic difference forms (QDFs), allowing the study of general dynamic quadratic supply rates. We work under the assumptions that the data-generating system is controllable, and an upper bound is given on its lag. As our main results, we provide sufficient conditions for the data to be informative for dissipativity. We also show that for a specific class of static supply rates, these conditions are both necessary and sufficient. For the latter supply rates, it turns out that certification of dissipativity is only possible from data that enable unique system identification. As auxiliary results, we highlight some properties of QDFs, such as upper bounds on the degree of storage functions.
\end{abstract}




\end{frontmatter}



\section{Introduction}

\noindent 
Dissipativity has a long history in systems and control theory, originating from the seminal works~\citep{willems1972dissipative,willems1972dissipativeII}. At its core, it formalizes the principle that a physical system cannot store more energy than what was supplied to it in the past. Mathematically, dissipativity is characterized by a \textit{supply rate}, quantifying the external energy supplied to the system, a \textit{storage function}, representing the internal energy stored in the system, and a \textit{dissipation inequality}, ensuring that any excess energy is dissipated.
Beyond its physically motivated origins, dissipativity provides a bridge between Lyapunov stability theory and input-output analysis, making it a key tool for analyzing and designing complex interconnected dynamical systems. Thus, dissipation inequalities provide a systematic approach to stability analysis of interconnections of \textit{open} systems, often leading to stability conditions that can be verified via linear matrix inequalities (LMIs)~\citep{boyd1994linear}. Furthermore, dissipativity offers a unifying framework and powerful computational tools for convergence analysis, forming the foundation of modern robust control~\citep{zhou1996robust, megretski1997system, vinnicombe2001uncertainty} and optimization methods~\citep{lessard2016analysis, lessard2022analysis,scherer2023optimization, doerfler2024towards,Eising2024,Karakai2025}. Consequently, it plays a fundamental role in numerous theoretical developments and numerical methods, spanning optimal control~\citep{vinter2010optimal}, game theory~\citep{bacsar1998dynamic}, and physics~\citep{van2014port}, as well as model predictive control~\citep{camacho2013model} and integral quadratic constraints~\citep{megretski1997system}. Originally introduced for general nonlinear state-space systems~\citep{willems1972dissipative}, dissipativity has also been extensively studied in the context of linear time-invariant (LTI) systems, including its well-established role in passivity theory~\citep{willems1972dissipativeII} 
and network analysis~\citep{anderson1972network}, as well as its formulation in behavioral system theory~\citep{polderman1998introduction,willems2007behavioral}, where storage functions and supply rates are naturally expressed  by quadratic differential/difference forms (QDFs)~\citep{willems1998quadratic}. 

Dissipativity is usually verified (or enforced) through dissipation inequalities that assume prior knowledge of a system model~\citep{van1996l2}. However, explicit system models are often either unavailable or difficult to obtain, leaving measured data as the primary source of information \citep{DBLSCT2025}. This has recently motivated a growing interest towards data-driven approaches for verifying dissipativity properties, aiming to bypass the need for an explicit system description. These approaches, however, often impose restrictive structural assumptions on system complexity and fail to accommodate general dynamic supply rates, which are fundamental, e.g., in applications involving integral quadratic constraints. To address these limitations, we use the framework of behavioral systems theory~\citep{polderman1998introduction,willems2007behavioral} to obtain conditions under which dissipativity of LTI systems can be verified from input-output data. Specifically, under the assumption that the data originates from a controllable LTI system whose lag is upper bounded by a given constant, we use QDFs to express dissipation inequalities purely in terms of data matrices constructed from a measured trajectory, eliminating the need for system identification. Surprisingly, our results establish that for a class of supply rates, dissipativity can only be ascertained if the data are rich enough for unique system identification.


\subsection{Related Work}

Several recent works have addressed the problem of verifying dissipativity properties directly from data. 
A seminal work connecting data, Lyapunov stability, and dissipativity properties in the context of behavioral system theory is~\citep{maupong2017dissipativity}, which develops methods to infer a dissipativity property — specifically, $L$-dissipativity — and construct Lyapunov functions directly from data. Building on and extending this work, the papers~\citep{romer2019oneshot,Rosa2021} further study methods for verifying $L$-dissipativity of LTI systems using a single input-output trajectory, bypassing the need for an explicit system model. Our approach differs in that we establish dissipativity conditions for LTI systems, whereas $L$-dissipativity only guarantees that the dissipation inequality holds over a finite horizon. In general, $L$-dissipativity does not imply dissipativity.

The paper~\citep{koch2020verifying} takes a different approach to verifying dissipativity of LTI systems directly from data. The method accounts for noisy measurements but requires access to the state of the system. In~\citep{koch2021provably}, these results are extended to a framework for verifying dissipativity properties in the more challenging setting of noisy input-output data. The work~\citep{vanwaarde2022dis} further establishes necessary and sufficient conditions for inferring dissipativity properties from noisy input-state data, leveraging the matrix S-lemma developed in~\citep{vanwaarde2020noisy}. In contrast to these approaches, our method requires neither precise knowledge of the state-space dimension nor direct state measurements. Instead, it relies on conditions verifiable from input-output data, assuming only that the data-generating system is controllable and that an upper bound is given on its lag. 

In~\citep{padoan2023data}, the authors derive representation criteria for conical, convex, and affine models in behavioral systems theory, using data matrices to construct shift-invariant models that satisfy non-negativity constraints, such as dissipation inequalities expressed by complementarity relations. The paper~\citep{verhoek2024data} develops methods to verify dissipativity in LPV systems from a single input-scheduling-output trajectory, formulating a semidefinite program based on an LMI to assess $(Q,S,R)$-type dissipativity, while leveraging structural information like scheduling rate bounds.  


Closely related to our work is~\citep{camlibel2024}, which studies system identification from the perspective of data informativity \citep{vanWaarde2020di}. Specifically, this paper provides necessary and sufficient conditions under which a minimal LTI system can be uniquely identified from finite input-output data (modulo state-space transformations), assuming given bounds on its lag and state-space dimension. We build on these tools to establish informativity conditions for dissipativity of LTI systems expressed through QDFs. We note that QDFs have already been successfully applied to solve other problems in data-driven control, such as stabilization using input-output data, see \citep{van2023behavioral} and \citep{Wakaiki2026}.

Also relevant are the fundamental results of~\citep{trentelman1997every} and its discrete-time counterpart~\citep{kaneko2003}, which show that every storage function is necessarily a function of state. Drawing inspiration from these works, Theorem~\ref{theoremequivdiss} establishes a related result with a new proof that directly constructs a storage function of state from a QDF storage function.

\subsection{Contributions}

The paper introduces a new data-based condition for dissipativity of LTI systems. Our main result, Theorem~\ref{t:dddiss}, is two-fold. The first part of this theorem states sufficient conditions under which the dissipativity of an LTI system can be verified from input-output data, for general dynamic supply rates described by QDFs. These conditions involve i) a rank condition on a data Hankel matrix, which ensures that the data enable unique system identification \citep{camlibel2024}, and ii) a data-based linear matrix inequality that certifies dissipativity of the data-generating system. The second part of Theorem~\ref{t:dddiss} asserts that these two conditions are also \emph{necessary} for static supply rates satisfying an inertia condition. Examples of such supply rates include the well-known $\ell_2$-gain and passivity supply rates. Therefore, a key insight is that for these static supply rates, identifiability is necessary for certifying dissipativity from data. Beyond our main result, we derive additional findings of independent practical interest, including an upper bound on the degree of the QDF storage function (Theorem~\ref{t:degreestorage}) and conditions under which a QDF storage function is a quadratic function of the system state, also providing a constructive method for deriving such a function (Theorem~\ref{theoremequivdiss}). 

\subsection{Paper Organization}

Section~\ref{sec:preliminaries} introduces preliminary notions and results on dissipativity in discrete-time state-space systems. Section~\ref{sec:behaviors_QDFs_dissipativity} recalls key concepts from behavioral systems theory, presenting both known results on behaviors, QDFs, and dissipativity, as well as clarification on the connection between state-space and behavioral formulations of dissipativity. Section~\ref{sec:problem_formulation} formalizes the problem of data-driven dissipativity certification. Section~\ref{sec:main_result} presents our main theoretical contribution. Finally, Section~\ref{sec:conclusion} summarizes our findings and discusses future research directions.

\subsection{Notation}
The set of integers is denoted by $\mathbb{Z}$ and the subset of non-negative integers by $\Zp$. Given $i,j\in\mathbb{Z}$, with $i\leq j$, we use $[i, j]$ to denote the discrete interval $[i, j]\cap \mathbb{Z}$. By convention, $[i,j] = \varnothing$ if $i > j$.

The \emph{image} and \emph{kernel} of the matrix $M$ are denoted by $\im M$ and $\ker M$, respectively. The $n \times n$ \emph{identity matrix} is denoted by $I_n$ and the $n \times m$ \emph{zero matrix} by $0_{n,m}$. We also use the shorthand notation $0_n := 0_{n,1}$. We sometimes omit the subscripts when the sizes of these matrices are clear from the context. 
The set of $n \times n$ \emph{real symmetric matrices} is denoted by $\S^n$. Let $a,b$ and $c$ denote the number of negative, zero, and positive eigenvalues of a matrix $M \in \S^n$, respectively. We then say that $M$ has \emph{inertia} $(a,b,c)$. Moreover, we write $M \geq 0$ when $M$ is \emph{positive semidefinite}, and $M > 0$ when $M$ is \emph{positive definite}.

A \emph{void matrix} is a matrix with zero rows and/or zero columns. We denote by $0_{n,0}$ and $0_{0,m}$, respectively, the $n \times 0$ and $0 \times m$ void matrices. If $M$ and $N$ are, respectively, $p \times q$ and $q \times r$ matrices, $MN$ is a $p \times r$ void matrix if $p = 0$
or $r = 0$ and $MN = 0_{p,r}$ if $p, r \geq 1$ and $q = 0$. The rank of a void matrix is defined as zero. Moreover, by convention, the matrix $0_{0,0}$ is symmetric and positive semidefinite. Furthermore, $\S^0 = \{0_{0,0}\}$ and $\mathbb{R}^0 = \{0_{0}\}$. 

Given sets $X$ and $Y$, $Y^X$ denotes the set of all functions $f:X \to Y$. Given a function $w: \Zp\to\R^q$ and $i,j\in\Zp$ with $i\leq j$, we define the \emph{restriction} of $w$ to the interval $[i,j]$ as 
$$
{w_{[i,j]}  =   [w(i)^\top\;w(i+1)^\top\cdots \; w(j)^\top  ]^\top}.
$$
In case $i>j$ we say that $w_{[i,j]} = 0_{0}$, by convention. Finally, given positive integers $k$ and $T$ with $k \leq T$, and the vector $w_{[0,T-1]}$, we define the \emph{Hankel matrix} of depth $k$ as 
    $$
    \calH_k(w_{[0,T-1]}) := \begin{bmatrix}
    w(0) & w(1) & \cdots & w(T-k)\\
    w(1) & w(2) & \cdots & w(T-k+1) \\
    \vdots & \vdots & & \vdots \\
    w(k-1) & w(k) & \cdots & w(T-1)
    \end{bmatrix}.
    $$
    
\section{Preliminaries}\label{sec:preliminaries}

\subsection{State-space systems and related matrices}
\label{subsec:state-space systems}

Consider a discrete-time linear time-invariant system described by the equations
\begin{equation}
    \label{sys}
\begin{aligned}
x(t+1) &= A x(t) + B u(t), \\
y(t)   &= C x(t) + D u(t),
\end{aligned}
\end{equation}
where $x(t) \in \mathbb{R}^n$ is the state, $u(t) \in \mathbb{R}^m$ is the input, and $y(t) \in \mathbb{R}^p$ is the output for $t 
\in \Zp$. Here, $n \geq 0$, $m \geq 1$ and $p \geq 1$. Moreover, $A \in \mathbb{R}^{n \times n}$, $B \in \mathbb{R}^{n \times m}$, $C \in \mathbb{R}^{p \times n}$ and $D \in \mathbb{R}^{p \times m}$. For $k \geq 0$, we define the $k$-th \emph{observability matrix} as
$$
\mathcal{O}_k := \begin{cases}
    0_{0,n} & \text{if } k = 0 \\
    \begin{bmatrix}
        \mathcal{O}_{k-1} \\ CA^{k-1} 
    \end{bmatrix} & \text{if } k \geq 1.
\end{cases}
$$
We denote the smallest integer $k \geq 0$ such that $\rank \mathcal{O}_k = \rank \mathcal{O}_{k+1}$ by $\ell(C,A)$. This integer is called the \emph{lag} of the pair $(C,A)$. Note that $0 \leq \ell(C,A) \leq n$. If $n = 0$ then also $\ell(C,A) = 0$. We call the pair $(C,A)$ \emph{observable} if $\rank \mathcal{O}_{n} = n$. 

Similarly, for $k \geq 0$, we define the $k$-th \emph{controllability matrix} by 
$$
\mathcal{C}_k := \begin{cases}
    0_{n,0} & \text{if } k = 0 \\
    \begin{bmatrix}
        A^{k-1} B & \mathcal{C}_{k-1}
    \end{bmatrix}& \text{if } k \geq 1.
\end{cases}
$$
We call the pair $(A,B)$ \emph{controllable} if $\rank \mathcal{C}_n = n$. Finally, for $k \geq 0$, we define the $k$-th \emph{Toeplitz matrix of Markov parameters} by
$$
\mathcal{T}_k := \begin{cases}
    0_{0,0} & \text{if } k = 0 \\
    \begin{bmatrix}
        \mathcal{T}_{k-1} & 0 \\ C \mathcal{C}_{k-1} & D
    \end{bmatrix} & \text{if } k \geq 1.
\end{cases}
$$

\subsection{Dissipativity of discrete-time state-space systems}
\label{ssec:dissipativityiso}

The concept of \emph{dissipativity} plays a central role in this paper. For state-space systems of the form \eqref{sys}, this is defined as follows.  

\begin{definition}
\label{def:ssdissipativity} 
Let ${S\in \S^{m+p}}$. The system \eqref{sys} is said to be \emph{dissipative} with respect to the {\em supply rate}  
\begin{equation} 
\label{supply}
s(u,y)=\begin{bmatrix}  u\\y\end{bmatrix}^\top S \begin{bmatrix} u\\y\end{bmatrix},
\end{equation} 
if there exists a positive semidefinite matrix $P\in\S^{n}$ so that 
\begin{equation} 
\label{storage} 
V(x)  := x^{\top} P x
\end{equation} 
satisfies
\begin{equation}
\label{dispineq}
  V(x(t+1)) - V(x(t)) \leq s(u(t),y(t))
\end{equation} 
for all $t \in \Zp$ and all input-state-output trajectories $(u,x,y): \Zp \to \mathbb{R}^{m+n+p}$ of system \eqref{sys}.
\end{definition}

\noindent 
The function $V$ in \eqref{storage} is called a \emph{storage function}, and the inequality \eqref{dispineq} is known as the \emph{dissipation inequality}, see \citep{willems1972dissipative}. It indicates that the energy stored at time $t+1$ cannot exceed the sum of the energy stored at time $t$ and the energy $s(u(t),y(t))$ supplied externally. As a result, internal ``creation of energy'' is precluded, allowing only for dissipation of energy. Well-known special cases \citep{van2000l2} of dissipativity include \emph{passivity}, which corresponds to the case that $m = p$ and 
\begin{equation}\label{eq:passive}
S = \begin{bmatrix}
    0 & I_m \\ 
    I_m & 0
\end{bmatrix},
\end{equation}
and \emph{finite $\ell_2$-gain} (of at most $\gamma \geq 0$), in which 
\begin{equation}\label{eq:L2gain}
S = \begin{bmatrix}
    \gamma^2 I_m & 0 \\ 0 & -I_p
\end{bmatrix}.
\end{equation}

\section{Behaviors, quadratic difference forms, and dissipativity} \label{sec:behaviors_QDFs_dissipativity}

The focus of this paper is on directly inferring dissipativity properties from data, integrating principles and tools from behavioral systems theory  \citep{willems1998quadratic,kaneko2003}. This requires key concepts from behavioral systems theory, such as the notions of behavior, quadratic difference forms, and a behavioral analogue of the notion of dissipativity, which we introduce in the following subsections.

\subsection{Discrete-time LTI dynamical systems}

In the language of behavioral systems theory, a \textit{dynamical system} (or, briefly, a \textit{system}) is a triple $(\T,\W,\B),$ where $\T$ is the \textit{time axis}, $\W$ is the \textit{signal set}, and $\B \subseteq (\W)^{\T}$ is the \textit{behavior} of the dynamical system. Throughout the paper, we focus on discrete-time LTI systems. In particular, we consider systems for which the time axis is ${\T=\Zp}$, the signal set is the linear space ${\W=\mathbb{R}^q},$ with $q = m+p$, and the (input-output) behavior $\B$ of the system is defined as  
\begin{equation}\label{behavior}
\B := \{ w = (u,y) \in (\mathbb{R}^q)^\Zp \mid \exists \, x \in (\mathbb{R}^n)^\Zp \text{ such that } \eqref{sys} \text{ holds for all } t \in \Zp \}.
\end{equation}
\noindent
Note that, throughout the paper, we consider a fixed partition of the variables $w$ into inputs $u$ and outputs $y$.

Given a behavior $\B$ of the form \eqref{behavior}, we say that \eqref{sys} (or simply, $(A,B,C,D)$) is a \emph{state-space representation} of $\B$. Clearly, $\B$ has many different state-space representations. We say a state-space representation $(A,B,C,D)$ of $\B$ is \emph{minimal} if its state-space dimension is smaller than or equal to the state-space dimension of any other state-space representation of $\B$. It can be shown that a minimal state-space representation $(A,B,C,D)$ is always \emph{observable}, i.e., the pair $(C,A)$ is observable (but $(A,B)$ need not be controllable) \citep{willems1986}. We denote the state-space dimension of any minimal state-space representation of $\B$ by $n(\B)$ and omit the dependence on $\B$ whenever it is clear from the context. The state-space dimension is an \textit{invariant} of the system, that is, it does not depend on the choice of minimal state-space representation. The \emph{lag} $\ell(\B)$ of the behavior $\B$ is defined as $\ell(C,A)$, where $(A,B,C,D)$ is any state-space representation of $\B$. It can be shown that the lag is also an invariant of the system, which does not depend on the particular choice of state-space representation. When no confusion can arise, we simply write $\ell$ instead of $\ell(\B)$. 

The following definition of controllability is adapted from \citep{Willems1991} to account for the time axis $\mathbb{T} = \Zp$. Specifically, the behavior $\B$ is called \emph{controllable}  if for all $w_1,w_2 \in \B$ and all $t_1 \in \mathbb{Z}_+$, there exists $w \in \B$ and $t_2 \geq t_1$ such that 
$$
w(t) = \begin{cases} w_1(t) & \text{ for } t \in [0,t_1] \\ w_2(t-t_2) & \text{ for } t \geq t_2. \end{cases}
$$
If $(A,B,C,D)$ is a minimal state-space representation of $\B$, then $\B$ is controllable if and only if the pair $(A,B)$ is controllable in the sense of Kalman. 

Finally, two state-space representations $(A,B,C,D)$ and $(A',B',C',D')$ are called \emph{isomorphic} if there exists a nonsingular matrix $T$ such that $A = TA'T^{-1}$, $B = TB'$, $C = C' T^{-1}$ and $D = D'$. The transformation associated with $T$ is referred to as a \textit{similarity transformation}. Note that isomorphic state-space representations give rise to the same input-output behavior.

\subsection{Quadratic difference forms}

We now review basic material on quadratic difference forms (QDFs) \citep{willems1998quadratic,kaneko2003}, which play a key role in defining the notion of dissipativity in the framework of behavioral systems theory.

Let $M \geq -1$ and $q \geq 1$ be integers. Define the matrix $\Psi \in \mathbb{S}^{(M+1)q}$ by 
$$
\Psi := \begin{cases}
    0_{0,0} & \text{if } M = -1 \\ 
    \begin{bmatrix}
\Psi_{0,0} & \cdots & \Psi_{0,M} \\
\vdots & \ddots & \vdots \\
\Psi_{M,0} & \cdots &\Psi_{M,M}
\end{bmatrix} & \text{if } M \geq 0,
\end{cases}
$$
where $\Psi_{i,j} \in \mathbb{R}^{q \times q}$ is such that $\Psi_{i,j} = \Psi_{j,i}^\top$ for $i,j \in [0,M]$. 

\begin{definition}
The quadratic difference form (QDF) associated with the matrix $\Psi \in \S^{(M + 1)q}$ is the operator 
  $Q_{\Psi}: (\mathbb{R}^q)^{\Zp} \to (\mathbb{R})^{\Zp}$
defined by
\begin{equation}  \label{ch0:eq:QDF} 
Q_{\Psi}(w)(t) := \begin{cases}
    0 & \text{if } M = -1 \\ 
    \sum_{i,j=0}^{M} w(t + i)^\top \Psi_{i,j} ~w(t +j) & \text{otherwise}.
\end{cases}
\end{equation}
\end{definition}

\noindent 
The QDF $Q_\Psi$ can be written in terms of $\Psi$ as 
\[
Q_{\Psi}(w)(t) = 
w_{[t,t+M]}^\top \Psi w_{[t,t+M]}.
\]
If $M = -1$, we define the {\em degree} of $Q_{\Psi}$ as $-1$. If $M \geq 0$, the degree of $Q_\Psi$ is defined as the smallest integer $d \geq -1$ such that $\Psi_{i,j} = 0$ for all $i > d$ and all $j$. The degree of $Q_{\Psi}$ is denoted by $\deg(Q_{\Psi})$. The matrix $\Psi$ is called a \textit{coefficient matrix} of the QDF. Note that a given QDF does not determine the coefficient matrix uniquely. However, if the degree of the QDF is $d$, it has a coefficient matrix $\Psi$ of size $(d + 1)q \times (d+1)q$. In this case, we call $\Psi$ a \emph{minimal} coefficient matrix of $Q_\Psi$. A QDF $Q_{\Psi}$ is referred to as \textit{static} if $\deg(Q_{\Psi}) = 0$, and \textit{dynamic} if $\deg(Q_\Psi) > 0$.

The QDF $Q_\Psi$ is called \emph{nonnegative} if $Q_{\Psi}(w) \geq 0$ for all $w: \Zp \rightarrow \mathbb{R}^q$. We denote this as $Q_{\Psi} \geq 0$. Clearly, this holds if and only if $\Psi \geq 0$. 
Likewise, we define \emph{nonpositivity}, which is denoted as $Q_\Psi \leq 0$.

For a given QDF $Q_\Psi$, its {\em rate of change} along a given $w: \Zp\rightarrow \mathbb{R}^q$ is given by
$$Q_\Psi(w)(t+1) - Q_\Psi(w)(t).
$$
The rate of change turns out to be a QDF itself. Indeed, by defining the matrix 
$\nabla \Psi \in \S^{(M + 2)q}$ by
\begin{equation} \label{ch0:eq:nabla phi}
\nabla \Psi := \begin{bmatrix} 0_{q,q} & 0 \\
                               0  &  \Psi \end{bmatrix} -
                     \begin{bmatrix} \Psi & 0 \\
                               0  &  0_{q,q} \end{bmatrix},                              
\end{equation}
it is easily verified that 
\[
Q_{\nabla \Psi}(w)(t) = Q_\Psi(w)(t+1) - Q_\Psi(w)(t)
\]
for all $w: \Zp \rightarrow \mathbb{R}^q$ and $t \in \Zp$.

QDFs are particularly relevant in combination with behaviors defined by LTI systems. Let $\B$ be the behavior defined in \eqref{behavior}. The QDF $Q_\Psi$ is called \emph{nonnegative on $\B$} if $Q_{\Psi}(w) \geq 0$ for all $w \in \B$. We denote this as $Q_\Psi \geq 0$ on $\B$. Likewise, we define \emph{nonpositivity} on $\B$ which is denoted by $Q_\Psi \leq 0$ on $\B$. For two given QDFs $Q_{\Psi_1}$ and $Q_{\Psi_2}$ we sometimes use the notation $Q_{\Psi_1} \leq Q_{\Psi_2}$ on $\B$, meaning that $Q_{\Psi_1} - Q_{\Psi_2} \leq 0$ on $\B$. Finally, $Q_{\Psi_1}$ and $Q_{\Psi_2}$ are called {\em $\B$-equivalent} if $Q_{\Psi_1}(w) = Q_{\Psi_2}(w)$ for all $w \in \B$.

\subsection{Dissipativity using quadratic difference forms}

\noindent
In this section, we define and characterize a behavioral notion of dissipativity. As before, throughout this section we consider the input-output behavior $\B$ as in \eqref{behavior}. In addition, consider a matrix $\Phi \in \S^{(M+1)q}$ defining the QDF $Q_\Phi$. We will refer to $Q_\Phi$ as the \emph{supply rate}. Note that the supply rate is in general dynamic, that is, it depends on the signal $w$ over a time window of length $\deg(Q_{\Phi})+1$.   
\begin{definition}
\label{def:behavioraldissipativity}
The behavior $\B$ is called \emph{dissipative with respect to the supply rate $Q_\Phi$} if there exists a quadratic difference form $Q_\Psi$ such that 
\begin{equation}
\label{dissipationineq}
Q_\Psi \geq 0 \text{ on } \B \text{ and } Q_{\nabla \Psi} \leq Q_\Phi \text{ on } \B.
\end{equation}
\end{definition}

A quadratic difference form $Q_\Psi$ satisfying \eqref{dissipationineq} is called a \emph{storage function} for the pair ($\B,Q_\Phi)$. In principle, dissipativity does not require positive semidefiniteness of the coefficient matrix $\Psi$ of the storage function (only nonnegativity of $Q_\Psi$ on $\B$ is required). Nonetheless, it turns out that we can take $\Psi \geq 0$ without loss of generality, as made precise in the next lemma.  

\begin{lemma}
\label{l:psdcoefficient}
The behavior $\B$ is dissipative with respect to the supply rate $Q_\Phi$ if and only if there exists a symmetric matrix $\Psi \geq 0$ such that
\begin{equation}\label{eq:dissipation_nablapsi}
 Q_{\nabla \Psi} \leq Q_\Phi \text{ on } \B.
\end{equation}
\end{lemma}
For the proof, we need the notion of \emph{restricted behavior}. Given $t_0,t_1 \in \Zp$, with $t_0 \leq t_1$, the restricted behavior is defined as
$$
\rB{t_0}{t_1} \: := \left\{
w_{[t_0,t_1]} \in \mathbb{R}^{(t_1-t_0+1)q} \mid w \in \B \right\}.
$$
Obviously, the restricted behavior is a subspace of $\mathbb{R}^{(t_1-t_0+1)q}$. 

\begin{proof}[Proof of Lemma~\ref{l:psdcoefficient}]
The `if' part is clear by noting that $\Psi \geq 0$ implies that $Q_\Psi \geq 0$ on $\B$. To prove the `only if'  part, assume that $\B$ is dissipative with respect to $Q_\Phi$. Let $Q_{\Psi'}$ be a storage function for $(\B,Q_\Phi)$ and denote its degree by $d$. Let the columns of the matrix $G$ form a basis for $\rB{0}{d}$. Note that by time-invariance, $\rB{t}{t+d} \subseteq \rB{0}{d}$ for all $t \in \Zp$ \citep{willems1986}. As such, $Q_{\Psi'} \geq 0$ on $\B$ implies that
$$
G^\top \Psi' G \geq 0,
$$
meaning that $G^\top \Psi' G = F^\top F$ for some real matrix $F$. 
Since $\ker G \subseteq \ker F$, we have that $\im F^\top \subseteq \im G^\top$, i.e., $F^\top = G^\top K^\top$ for some matrix $K$. This implies that $G^\top \Psi' G = G^\top K^\top K G$. Now define $\Psi := K^\top K \geq 0$. By construction, $Q_{\Psi'}$ and $Q_\Psi$ are $\B$-equivalent. Therefore, $Q_{\nabla \Psi} \leq Q_\Phi$ on $\B$. This proves the lemma. 
\end{proof}

So far, we have not specified the degree of the storage function, which could be large in practice. Nevertheless, we can prove bounds on the degree of the storage function $Q_\Psi$ in terms of the degree of the supply rate $Q_\Phi$ and the lag of $\B$. 

\begin{theorem}
\label{t:degreestorage} Consider a behavior $\B$ of the form \eqref{behavior}. Assume that $\B$ is dissipative with respect to $Q_\Phi$. Then there exists a storage function for $(\B,Q_\Phi)$ of degree less than 
$$
\max\{\deg(Q_\Phi),\ell(\B)\}.
$$
\end{theorem}

\noindent 
The proof of Theorem~\ref{t:degreestorage} is provided in \ref{sec:proofdegreestorage}.

\subsection{Relation between state-space and behavioral dissipativity notions}

In this subsection we relate the behavioral notion of dissipativity of Definition~\ref{def:behavioraldissipativity} with the notion of dissipativity of state-space systems in Definition~\ref{def:ssdissipativity}.  

Let $\B$ be the input-output behavior as defined in \eqref{behavior}. Consider the (static) supply rate \eqref{supply}, with $S \in \mathbb{S}^q$, and the QDF $Q_S$.  Note that we have
$$
Q_S(w)(t) = s(u(t),y(t)).
$$ 
The question we address is:  how does dissipativity of the state-space system \eqref{sys} relate to dissipativity of its corresponding input-output behavior $\B$? 
We emphasize that the ``state-space'' notion of dissipativity is defined in terms of a storage function  $V(x) =  x^\top P x$ that is a quadratic function of the \emph{internal} (state) variables of system~\eqref{sys}. On the other hand, the ``behavioral'' notion of dissipativity of $\B$ is defined in terms of storage function that is a QDF of \emph{external} (input-output) variables. 
The question posed above is thus closely connected to another important question: when is a QDF storage function a (quadratic) function of the state? This question has been extensively studied and answered in the seminal works~\citep{trentelman1997,kaneko2003}, which lay the groundwork for our next theorem. While our result may be derived from those of \citep{kaneko2003}, we present an alternative (constructive) proof, presented in \ref{sec:proofstatestorage}. The proof not only clarifies the relationship between the ``state-space'' and ``behavioral'' notions of dissipativity, but also provides a formula to explicitly construct a positive semidefinite matrix $P$, defining a storage function of state, given a QDF storage function.

\begin{theorem}
\label{theoremequivdiss} 
Consider the behavior $\B$ in \eqref{behavior} and suppose that \eqref{sys} is a minimal state-space representation of $\B$. Then \eqref{sys} is dissipative with respect to the supply rate $s(u,y)$ if and only if $\B$ is dissipative with respect to $Q_S$.
Moreover, if $Q_\Psi$ is a storage function for the pair $(\B,Q_S)$, where $\Psi \in \mathbb{S}^{qd}$ with $d\in\Zp$  and $\Psi \geq 0$, then it admits the form
\begin{equation}
    \label{QPsiandP}
Q_\Psi(w)(t) = x(t)^\top P x(t),
\end{equation}
for all $w\in\B$, where $x:\Zp \to \R^n$ is the state trajectory of \eqref{sys} corresponding to $w$, and $P \in \S^n$ is defined as
\begin{equation}
    \label{P} 
P := \begin{bmatrix}
    0 \\ \mathcal{O}_d
\end{bmatrix}^\top \Pi^\top \Psi \Pi \begin{bmatrix}
    0 \\ \mathcal{O}_d
\end{bmatrix} \geq 0.
\end{equation}
Here $\Pi \in \R^{qd \times qd}$ is the permutation matrix such that 
\begin{equation}
    \label{permutation}
 w_{[t,t+d-1]} = \Pi \begin{bmatrix}
    u_{[t,t+d-1]} \\ y_{[t,t+d-1]}
\end{bmatrix}
\end{equation}
for all $w \in \B$ and $t \in \Zp$. 
\end{theorem}

\section{The problem of data-driven dissipativity}  \label{sec:problem_formulation}

We denote by $\Lmp$ the set of all behaviors of the form \eqref{behavior} with $m$ inputs and $p$ outputs. Moreover, $\Lmpcont := \{ \B \in \Lmp \mid \B \text{ is controllable}\}$ denotes the subset of all controllable behaviors. Let $L \geq 0$, and define the set 
$$
\LmpLcont := \{ \B \in \Lmpcont \mid \ell(\B) \leq L\}
$$
of all controllable behaviors with $m$ inputs and $p$ outputs whose lag is bounded by $L$. 

In what follows, we will consider the true behavior $\B_\true$. We will suppose that this behavior is unknown, but we do assume that $\B_\true$ is controllable and a bound $L \geq 0$ on its lag is given. In other words, $\B_\true \in \LmpLcont$. The set $\LmpLcont$ thus captures our \emph{prior knowledge} on the true behavior. Our goal is to verify dissipativity properties of $\B_\true$ from measured input-output data produced by this behavior. 
 
To this end, consider the $T \geq 1$ input-output samples 
$$
w(t) = \begin{bmatrix}
    u(t) \\ y(t)
\end{bmatrix}
$$ 
for $t = 0,1,\dots,T-1$, obtained from the true behavior, i.e.,
$$
w_{[0,T-1]}
\in \B_\true\!\!\mid_{[0,T-1]}.
$$
The set of behaviors consistent with the data is given by
$$
\Sigma = \left\{\B \in \Lmp \mid w_{[0,T-1]} \in \rB{0}{T-1}\right\}.
$$
This set thus consists of all behaviors that could have generated the given input-output data. 

With this in place, we can now define the notion of informative data studied in this paper. Consider the supply rate $Q_\Phi$. Recall from Theorem~\ref{t:degreestorage} that, if a behavior $\B$ is dissipative with respect to $Q_\Phi$, then there always exists a storage function $Q_\Psi$ of degree less than $\max \{\deg(Q_\Phi),\ell(\B)\}$. Based on this observation, we define the data to be informative for dissipativity if there exists a \emph{common} storage function $Q_\Psi$ of degree less than $\max \{\deg(Q_\Phi),L\}$ \emph{for all behaviors} consistent with the data and the prior knowledge. This is formalized in the following definition. 

\begin{definition} 
Let $Q_\Phi : (\mathbb{R}^q)^{\mathbb{Z}_+} \to (\mathbb{R})^{\mathbb{Z}_+}$ be a QDF and define $d:= \max\{\deg(Q_\Phi),L\}$. We say that the data $w_{[0,T-1]}$ are \emph{informative for dissipativity with respect to $Q_{\Phi}$} if  there exists a positive semidefinite $\Psi \in \S^{qd}$ such that 
\begin{equation}
\label{dissipationineqdata}
Q_{\nabla \Psi} \leq Q_\Phi \text{ on } \B
\end{equation}
for all $\B \in \Sigma \cap \LmpLcont$. 
\end{definition}

We abbreviate and say $w_{[0,T-1]}$ are informative for dissipativity when the supply rate $Q_{\Phi}$ is clear from the context. The problem studied in this paper is to provide conditions under which a given data set is informative for dissipativity. In addition, our aim is to provide a recipe for computing a storage function $Q_\Psi$ whenever it exists. 

\section{Main results} \label{sec:main_result}

In this section, we will state the main results of this paper. Before doing so, we define the following integer, originally introduced in \citep{camlibel2024}. It is given here in behavioral terms as follows:

\begin{equation*}
    \nmin := \min \{ n(\B) \mid \B \in \Sigma \}.
\end{equation*}

Thus, $\nmin$ is the smallest state-space dimension of any system consistent with the data. As shown in \citep{camlibel2024}, $\nmin$ can be computed from the data using the rank of a finite number of data Hankel matrices of different depths. The following theorem is the main result of this paper. It provides a sufficient condition for informativity for dissipativity, and asserts that this condition is necessary in the case of static supply rates satisfying an inertia assumption. The result also provides a recipe to obtain storage functions directly from data by solving a linear matrix inequality.

\begin{theorem} 
\label{t:dddiss}
Let $Q_\Phi : (\mathbb{R}^q)^{\mathbb{Z}_+} \to (\mathbb{R})^{\mathbb{Z}_+}$ be a QDF with minimal coefficient matrix $\Phi$ and define $d:= \max\{\deg(Q_\Phi),L\}$. Let $\B_\true \in \LmpLcont$ and consider the data $w_{[0,T-1]} \in \rBtrue{0}{T-1}$, where $T \geq d + 1$. Define $H := \mathcal{H}_{d+1}(w_{[0,T-1]})$. Then the following statements hold. 
\begin{enumerate}[label=(\roman*)]
    \item \label{item1} If $\rank H = \nmin + m(d+1)$ and there exists a positive semidefinite $\Psi \in \S^{q d}$ such that 
\begin{equation} \label{eq:dissipativity_LMI}
H^\top \left( \begin{bmatrix}
\Psi & 0 \\ 0 & 0_{q , q}
\end{bmatrix} - \begin{bmatrix}
0_{q , q} & 0 \\ 0 & \Psi
\end{bmatrix} + \begin{bmatrix}
    \Phi & 0 \\ 0 & 0
\end{bmatrix} \right) H \geq 0
\end{equation}
then the data $w_{[0,T-1]}$ are informative for dissipativity. Moreover, in this case, $Q_\Psi$ is a storage function for $(\B_*,Q_\Phi)$.
\item \label{item2} Now assume, in addition, that $Q_\Phi$ is a static supply rate and $\Phi \in \mathbb{S}^q$ is partitioned as
\begin{equation}
    \label{Phiprime}
\Phi = \begin{bmatrix}
    \Phi_{11} & \Phi_{12} \\ \Phi_{21} & \Phi_{22}
\end{bmatrix},
\end{equation}
where $\Phi_{11} \in \mathbb{S}^m$, $\Phi_{12} = \Phi_{21}^\top \in \mathbb{R}^{m \times p}$ and $\Phi_{22} \in \mathbb{S}^p$. Suppose that $\Phi$ has inertia $(p,0,m)$, and $\Phi_{22} \leq 0$. Then the data $w_{[0,T-1]}$ are informative for dissipativity if and only if $\rank H = \nmin + m(L+1)$ and there exists a positive semidefinite $\Psi \in \S^{qL}$ satisfying \eqref{eq:dissipativity_LMI}.
\end{enumerate}
\end{theorem}

\begin{remark}
The rank condition in Theorem~\ref{t:dddiss}~\ref{item2} was shown to be necessary and sufficient for informativity for system identification in \citep[Thm. 9]{camlibel2024} in the situation that only a bound on the lag is given, see also \citep[Rem. 11]{camlibel2024}. Therefore, Theorem~\ref{t:dddiss}~\ref{item2} can be interpreted as follows: under the additional assumptions, the data $w_{[0,T-1]}$ can only be informative for dissipativity if they enable unique system identification, i.e., if $\Sigma \cap \LmpLcont = \{\B_\true\}$. We note that the passive and $\ell_2$-gain supply rates introduced in Section~\ref{ssec:dissipativityiso} are examples of supply rates satisfying the additional assumptions of Theorem~\ref{t:dddiss}~\ref{item2}.
\end{remark}

\begin{proof}[Proof of Theorem~\ref{t:dddiss}]
    Note that in the case that $\deg(Q_\Phi) = 0$, as assumed in \ref{item2}, we have $d = L$. Therefore, to prove the `if' statement of both \ref{item1} and \ref{item2}, suppose that $\rank H = \nmin + m(d+1)$ and there exists a positive semidefinite $\Psi \in \S^{qd}$ satisfying \eqref{eq:dissipativity_LMI}.  The rank condition implies, by \citep[Thm. 9 and Rem. 11]{camlibel2024}, that $\Sigma \cap \LmpLcont = \{\B_\ast\}$ and $\nmin = n(\B_\true)$. As such, $\im H = \B_\true\!\!\mid_{[0,d]}$. In fact, by time-invariance,   $\B_\true\!\!\mid_{[t,t+d]} \subseteq \im H$ for any $t \in \Zp$. We conclude from \eqref{eq:dissipativity_LMI} that $Q_{\nabla \Psi} \leq Q_\Phi$ on $\B_\ast$. This proves that the data $w_{[0,T-1]}$ are informative for dissipativity, and $Q_\Psi$ is indeed a storage function for $(\B_*,Q_\Phi)$.

    Next, we prove the `only if' part of item \ref{item2}. To this end, we suppose that $\Phi$ is of the form \eqref{Phiprime} with inertia $(p,0,m)$ and $\Phi_{22}\leq 0$. Assume that the data $w_{[0,T-1]}$ are informative for dissipativity. Suppose on the contrary that
    \begin{equation}
    \label{negationrank}
    \rank H < \nmin + m(L+1).
    \end{equation} 
    Let $(A_\ast,B_\ast,C_\ast,D_\ast)$ be a controllable and observable state-space representation of the true behavior $\B_\true$. Denote $n_\true = n(\B_\true)$ and $\ell_\true = \ell(\B_\true)$. We will now distinguish between two cases, namely $L = 0$ and $L > 0$. 
    
    We start with the case that $L = 0$. This implies that $\ell_\true = 0$ and thus $n_\true = 0$. Thus, the matrices $A_\true, B_\true$ and $C_\true$ are void. Hence, we obtain $\mathcal{H}_1(y_{[0,T-1]}) = D_\true \mathcal{H}_1(u_{[0,T-1]})$. Using the fact that $\nmin = 0$ because $n_\true = 0$, we conclude by \eqref{negationrank} that $\rank \mathcal{H}_1(u_{[0,T-1]}) < m$. Let $\eta \in \mathbb{R}^m$ be a vector such that $\|\eta\| = 1$ and $\eta^\top \mathcal{H}_1(u_{[0,T-1]}) = 0$. This implies that the input-output behavior $\B_{\alpha,\xi}$ associated with the state-space system $(A_\true,B_\true,C_\true,D_\true+\alpha \xi \eta^\top)$ is such that $\B_{\alpha,\xi} \in \Sigma \cap \LmpLcont$ for any $\xi \in \mathbb{R}^p$ and $\alpha \in \mathbb{R}$. By hypothesis, the data $w_{[0,T-1]}$ are informative for dissipativity, and thus $\B_{\alpha, \xi}$ is dissipative for all $\xi \in \mathbb{R}^p$ and $\alpha \in \mathbb{R}$. Therefore,
    \begin{equation}
    \label{ineqetaalphaxi}
    \begin{bmatrix}
        I \\ D_\true + \alpha \xi \eta^\top 
    \end{bmatrix}^\top
    \Phi 
    \begin{bmatrix}
        I \\ D_\true + \alpha \xi \eta^\top 
    \end{bmatrix} \geq 0 
    \end{equation} 
    for all $\xi \in \mathbb{R}^p$ and $\alpha \in \mathbb{R}$. If $\Phi_{22} \neq 0$, then we can choose $\xi$ such that $\xi^\top \Phi_{22} \xi < 0$ because $\Phi_{22} \leq 0$. In this case, $\alpha$ can be chosen sufficiently large, violating \eqref{ineqetaalphaxi} and thus leading to a contradiction to \eqref{negationrank}. On the other hand, if $\Phi_{22} = 0$ then \eqref{ineqetaalphaxi} reduces to 
    $$
        \begin{bmatrix} I \\ D_\true \end{bmatrix}^\top \Phi \begin{bmatrix} I \\ D_\true \end{bmatrix} + \alpha \left( \Phi_{12} \xi \eta^\top + \eta \xi^\top \Phi_{21} \right) \geq 0. 
    $$
    By the inertia condition on $\Phi$ and the hypothesis that $\Phi_{22} = 0$ we conclude that $\Phi_{12} \neq 0$. Therefore, we can choose $\xi$ such that $\Phi_{12} \xi \neq 0$ and thus the symmetric matrix $\Phi_{12} \xi \eta^\top + \eta \xi^\top \Phi_{21}$ is nonzero because $\eta \neq 0$. Therefore, we can choose $\alpha$ sufficiently large or sufficiently small, violating \eqref{ineqetaalphaxi}, which again leads to a contradiction to \eqref{negationrank}. 

    Next, we consider the case that $L > 0$. We claim that there exists a $\B \in \Sigma \cap \LmpLcont$ with lag \emph{precisely} $L = \ell(\B)$ such that
    \begin{equation}
    \label{rankineqB}
    \rank
  \begin{bmatrix}
        \mathcal{H}_1(x_{[0,T-L-1]}) \\
        \mathcal{H}_{L+1}(u_{[0,T-1]}) 
    \end{bmatrix} < n(\B)+m(L+1),
    \end{equation}
    where $x_{[0,T]}$ is a state sequence of some controllable and observable state-space representation $(A,B,C,D)$ of $\B$, compatible with the data $w_{[0,T-1]}$.
    
    To show this, consider the true behavior $\B_\true \in \Sigma \cap \LmpLcont$ and let $x^\true_{[0,T]}$ be a state sequence of $(A_\true,B_\true,C_\true,D_\true)$ compatible with the input-output data $w_{[0,T-1]}$. Because $\ell_\true \leq L$, \eqref{negationrank} implies that 
    
    \begin{equation}
        \label{rankineq}
    \rank
    \begin{bmatrix}
        \mathcal{H}_1(x^*_{[0,T-L-1]}) \\
        \mathcal{H}_{L+1}(u_{[0,T-1]}) 
    \end{bmatrix} = \rank H < \nmin + m(L+1) \leq n_\true+m(L+1).
    \end{equation}
    As such, if $\ell_\true = L$, then we can simply take $\B = \B_\true$, proving the claim.

    Next, consider the case that $\ell_\true < L$. Define $\mu := L - \ell_\true \geq 1$. Because of the strict inequality in \eqref{rankineq}, it follows from \citep[Lem. 28]{camlibel2024} that there exists a controllable $\B \in \Sigma$ with lag $\ell(\B) = \ell_\true + \mu = L$. Let $(A,B,C,D)$ be a minimal state-space representation of $\B$, and denote by $x_{[0,T]}$ a state sequence of this representation compatible with the input-output data $w_{[0,T-1]}$. We note that
    $$
    \rank \begin{bmatrix}
        \mathcal{H}_1(x_{[0,T-L-1]}) \\ \mathcal{H}_{L+1}(u_{[0,T-1]})
    \end{bmatrix} = \rank H < \nmin + m(L+1) \leq n(\B) + m(L+1),
    $$
    from which we conclude that there indeed exists a behavior $\B \in \Sigma \cap \LmpLcont$ with lag precisely $L$ satisfying \eqref{rankineqB}.
    
    Now, by hypothesis, there exists a positive semidefinite $\Psi \in \mathbb{S}^{qL}$ so that $Q_{\nabla \Psi} \leq Q_\Phi$ on $\B$. In what follows, we will abbreviate $n = n(\B)$. By \eqref{rankineqB}, there exist vectors $\xi \in \mathbb{R}^n$ and $\eta_i \in \mathbb{R}^m$ for $i = 0,1,\dots,L$, not all zero, such that
    $$
    \begin{bmatrix}
        \xi^\top & \eta_0^\top & \cdots & \eta_{L}^\top 
    \end{bmatrix}
    \begin{bmatrix}
        \mathcal{H}_1(x_{[0,T-L-1]}) \\
        \mathcal{H}_{L+1}(u_{[0,T-1]})
    \end{bmatrix}
    =0.
    $$
    Let $\zeta \in \mathbb{R}^n$ be a nonzero vector such that 
    \begin{equation}
        \label{CAzeta}
    CA^i \zeta = 0 
    \end{equation}
    for $i = 0,\dots,L-2$. This condition on $\zeta$ is void if $L = 1$. 
    Consider a nonzero $\alpha \in \mathbb{R}$ and define the matrices
    $$
    \hat{A}_\alpha := A + \alpha \zeta \xi^\top, \: \hat{B}_\alpha  := B + \alpha \sum_{j=0}^{L} \hat{A}_\alpha^j \zeta \eta_j^\top, \: \hat{C}_\alpha := C, \:  \text{ and }  \hat{D}_\alpha := D + \alpha \sum_{j=0}^{L-1} C \hat{A}_\alpha^j \zeta \eta_{j+1}^\top.
    $$
    It follows from Lemma~\ref{lemmacontrollability} in \ref{sec:lemcont} and the controllability of $(A,B)$ that there exists an $\bar{\alpha} > 0$ such that $(\hat{A}_\alpha,\hat{B}_\alpha)$ is controllable for all $\alpha \in \mathbb{R}$ satisfying $|\alpha| \geq \bar{\alpha}$. Let $\alpha$ be such a real number, and denote by $\hat{\B}_\alpha$ the behavior associated with $(\hat{A}_\alpha,\hat{B}_\alpha,\hat{C}_\alpha,\hat{D}_\alpha)$. Then, it follows from \citep[Lem. 24(a)]{camlibel2024} that $\hat{\B}_\alpha \in \Sigma$. Moreover, because $(A,B)$ is controllable, it follows from \citep[Lem. 24(b)]{camlibel2024} that $(A,B,C,D)$ and $(\hat{A}_\alpha,\hat{B}_\alpha,\hat{C}_\alpha,\hat{D}_\alpha)$ are not isomorphic. 

    \noindent  
    Now, by definition of $\zeta$ we have that 
$$ 
\mathcal{O}_{L} = \begin{bmatrix}
    C \\ CA \\ \vdots \\ CA^{L-1}
\end{bmatrix} = \begin{bmatrix}
    \hat{C}_\alpha \\ \hat{C}_\alpha \hat{A}_\alpha \\ \vdots \\ \hat{C}_\alpha \hat{A}_\alpha^{L-1}
\end{bmatrix}.
$$
By Theorem~\ref{theoremequivdiss}, the systems $(A,B,C,D)$ and $(\hat{A}_\alpha,\hat{B}_\alpha,\hat{C}_\alpha,\hat{D}_\alpha)$ are dissipative with respect to the supply rate
$$
s(u,y) = 
\begin{bmatrix}
    u \\ y
\end{bmatrix}^\top \Phi \begin{bmatrix}
    u \\ y
\end{bmatrix},
$$
and, by the same theorem, we conclude that
$$
P :=  \begin{bmatrix}
    0 \\ \mathcal{O}_{L}
\end{bmatrix}^\top \Pi^\top \Psi \Pi \begin{bmatrix}
    0 \\ \mathcal{O}_{L}
\end{bmatrix}
$$
is the coefficient matrix of a quadratic storage function of state for both $(A,B,C,D)$ and $(\hat{A}_\alpha,\hat{B}_\alpha,\hat{C}_\alpha,\hat{D}_\alpha)$. Here we recall that $\Pi$ is defined by \eqref{permutation}. By the hypothesis that $\Phi$ has inertia $(p,0,m)$ and the fact that $(C,A)$ is observable, we conclude that $P > 0$ by \citep[Lemma 4.4]{burohman2023thesis}. 

    We now distinguish the cases that $\xi$ is nonzero and zero. We start with $\xi \neq 0$. In this case, let $x(0) \in \R^n$ be such that $\xi^\top x(0) \neq 0$ and choose $u(0) = 0 $. 

We investigate the situation in which both systems $(A,B,C,D)$ and $(\hat{A}_\alpha,\hat{B}_\alpha,\hat{C}_\alpha,\hat{D}_\alpha)$ are initialized at the state $x(0)$, and are influenced by $u(0)$. The resulting states of these respective systems at time $t=1$ are:
\begin{align*}
    x(1) &= Ax(0) \\
    \hat{x}(1) &= x(1) + \alpha v,
\end{align*}
where $v := \zeta \xi^\top x(0)$ is a nonzero vector in $\R^n$ since $\zeta \neq 0$ and $\xi^\top x(0) \neq 0$. The outputs of both systems at time $t=0$ are given by 
\begin{align*}
    y(0) &= Cx(0) \\
    \hat{y}(0) &= y(0).
\end{align*}
The dissipation inequality \eqref{dispineq} for system $(\hat{A}_\alpha,\hat{B}_\alpha,\hat{C}_\alpha,\hat{D}_\alpha)$ at time $t = 0$ now reads: 
\begin{equation}
    \label{dispineqhat}
(x(1) + \alpha v)^\top P (x(1) + \alpha v) - x(0)^\top P x(0) \leq \begin{bmatrix}
    0 \\ y(0)
\end{bmatrix}^\top \Phi \begin{bmatrix}
    0 \\ y(0)
\end{bmatrix}.
\end{equation}
There is only one term in this inequality that is quadratic in $\alpha$, namely $\alpha^2 v^\top P v$ on the left hand side. Note that $\alpha^2 v^\top P v > 0$ since $P >0$ and $v \neq 0$. As such, there exists a sufficiently large 
$\alpha \geq \bar{\alpha}$ that violates the dissipation inequality \eqref{dispineqhat}. We thus arrive at a contradiction to \eqref{negationrank}.

Next, we consider the second case in which $\xi = 0$. In this case,
$$
\hat{A}_\alpha = A, \:\: \hat{B}_\alpha = B + \alpha \sum_{j=0}^{L} A^j \zeta \eta_j^\top, \:\: \hat{C}_\alpha = C, \:\: \text{ and } \:\: \hat{D}_\alpha = D + \alpha CA^{L-1} \zeta \eta_{L}^\top,
$$
where the expression of $\hat{D}_\alpha$ follows due to \eqref{CAzeta}. Within this case, we now make a distinction between the cases that $F:=\sum_{j=0}^{L} A^j \zeta \eta_j^\top$ is nonzero and zero. First consider $F \neq 0$. Then we choose $x(0) = 0$  and $u(0) \in \R^m$ such that $Fu(0) \neq 0$. The resulting states of $(A,B,C,D)$ and $(\hat{A}_\alpha,\hat{B}_\alpha,\hat{C}_\alpha,\hat{D}_\alpha)$ at time $t=1$ are:
\begin{align*}
    x(1) &= Bu(0) \\
    \hat{x}(1) &= x(1) + \alpha v,
\end{align*}
where $v:= Fu(0) \neq 0$ and the outputs at time $t = 0$ are:
\begin{align*}
    y(0) &= Du(0) \\
    \hat{y}(0) &= y(0) + \alpha z,
\end{align*}
where $z:= CA^{L-1}\zeta \eta_{L}^\top u(0)$. The dissipation inequality \eqref{dispineq} for system $(\hat{A}_\alpha,\hat{B}_\alpha,\hat{C}_\alpha,\hat{D}_\alpha)$ at time $t = 0$ now reads: 
\begin{equation}
    \label{dispineqhat2}
(x(1) + \alpha v)^\top P (x(1) + \alpha v) \leq \begin{bmatrix}
    u(0) \\ y(0) + \alpha z 
\end{bmatrix}^\top \Phi \begin{bmatrix}
    u(0) \\ y(0) + \alpha z 
\end{bmatrix}.
\end{equation}
In this case, the dissipation inequality has two quadratic terms in $\alpha$, namely $\alpha^2 v^\top P v > 0$ on the left hand side, and $\alpha^2z^\top \Phi_{22} z$ on the right hand side. Since $\Phi_{22} \leq 0$, we have $\alpha^2z^\top \Phi_{22} z \leq 0$. As such, there exists a sufficiently large 
$\alpha \geq \bar{\alpha}$ that violates the dissipation inequality. Hence, also in this case, we reach a contradiction to \eqref{negationrank}.

Next, suppose that $F = 0$. Since $\xi = 0$, we now have that
$$
\hat{A}_\alpha = A, \:\: \hat{B}_\alpha = B, \:\: \hat{C}_\alpha = C, \:\: \text{ and } \:\: \hat{D}_\alpha = D + \alpha G,
$$
where $G := CA^{L-1}\zeta \eta_{L}^\top$. because $(A,B,C,D)$ and $(\hat{A}_\alpha,\hat{B}_\alpha,\hat{C}_\alpha,\hat{D}_\alpha)$ are not isomorphic, we have that $G \neq 0$. Suppose that $\Phi_{22} G \neq 0$. Then select $u(0) \in \mathbb{R}^m$ such that $\Phi_{22} G u(0) \neq 0$. Let $x(0) = 0 \in \mathbb{R}^n$ and compute $x(1) = Bu(0)$ and $y(0) = Du(0)$. Then, the dissipation inequality for $(\hat{A}_\alpha,\hat{B}_\alpha,\hat{C}_\alpha,\hat{D}_\alpha)$ at time $t = 0$ reads: 
$$
x(1)^\top P x(1) \leq \begin{bmatrix}
    u(0) \\ y(0) + \alpha G u(0)
\end{bmatrix}^\top \Phi \begin{bmatrix}
    u(0) \\ y(0) + \alpha G u(0)
\end{bmatrix}.
$$
In this case, the only quadratic term in $\alpha$ is $\alpha^2 u(0)^\top G^\top \Phi_{22} G u(0) < 0$. Therefore, we can choose a sufficiently large 
$ \alpha \geq \bar{\alpha}$ to violate the dissipation inequality, leading again to a contradiction to \eqref{negationrank}.

Finally, suppose that $\Phi_{22} G = 0$. The hypotheses on $\Phi$ imply that $\Phi$ is nonsingular, and thus $\Phi_{12} G \neq 0$ because $G$ is nonzero. Since $G$ has rank one and $\Phi_{12} G \neq 0$ then  $\Phi_{12} G$ has rank one, and thus it is not skew-symmetric. Therefore, there exists a vector $u(0) \in \mathbb{R}^m$ such that $u(0)^\top \Phi_{12} G u(0) \neq 0$. For $x(0) = 0$, $x(1)=Bu(0)$ and $y(0)=Du(0)$ as before, the dissipation inequality for $(\hat{A}_\alpha,\hat{B}_\alpha,\hat{C}_\alpha,\hat{D}_\alpha)$ at time $t = 0$ now reads: 
$$
x(1)^\top P x(1) \leq \begin{bmatrix}
    u(0) \\ y(0)
\end{bmatrix}^\top \Phi \begin{bmatrix}
    u(0) \\ y(0) 
\end{bmatrix} + 2 \alpha u(0)^\top \Phi_{12} G u(0).
$$
Depending on the sign of $u(0)^\top \Phi_{12} G u(0)$, we can now choose a sufficiently large or sufficiently small $\alpha \in \mathbb{R}$ such that $|\alpha| \geq \bar{\alpha}$, and for which the dissipation inequality is violated. This again leads to a contradiction to \eqref{negationrank}. 

We therefore see that in all cases \eqref{negationrank} is violated, allowing us to conclude that $\rank H = \nmin + m(L+1)$. This rank condition implies by \citep[Thm. 9]{camlibel2024} that $\Sigma \cap \LmpLcont = \{\B_\true\}$ and $\nmin = n_\true$. As such, we conclude that $\im H = \B_\ast\!\!\mid_{[0,L]}$. Since $\Psi \geq 0$ is such that $Q_{\nabla \Psi} \leq Q_\Phi$ on $\B_*$, the matrix $\Psi$ satisfies \eqref{eq:dissipativity_LMI}. This proves the theorem.
\end{proof}

\section{Conclusion} \label{sec:conclusion} 

The paper has established conditions on input-output data for informativity for dissipativity, without requiring an explicit system model. Assuming the data-generating system is LTI and controllable and its lag is upper bounded by a given constant, we have formulated dissipation inequalities in terms of QDFs and data matrices using behavioral systems theory, enabling a purely data-driven characterization of dissipativity.  


Our main result has established sufficient conditions for informativity for dissipativity in the case of general dynamic supply rates. For static supply rates satisfying an inertia assumption, we have further shown that this condition is also necessary. In this case, a key insight is that dissipativity can only be ascertained from data that are informative for system identification. As byproducts of our development, we have also presented auxiliary results on quadratic difference forms. In particular, we have shown that the degree of a QDF storage function is always less than the maximum of the lag of the system and the degree of the supply rate. In addition, we have shown that for static supply rates, the QDF storage function is always a quadratic function of the state, and we have provided an explicit formula for the coefficient matrix of this state function. 

Our findings provide a rigorous foundation for data-driven dissipativity analysis, but also raise several further questions. A fundamental open question is to provide necessary and sufficient conditions for data-driven dissipativity for general dynamic supply rates. Another promising avenue, which is the subject of ongoing work, is the application of these conditions to the study of optimization algorithms viewed as discrete-time dynamical systems—a step toward formal guarantees for algorithmic pipelines and their convergence properties. Finally, extensions to nonlinear systems and noise-robust formulations are also natural directions for future research. 

\appendix

\section{Relation between state and input-output variables} \label{sec:lemstate}

In what follows, we present an auxiliary lemma relating the state and input-output variables of an LTI system. 

\begin{lemma}
\label{l:Z1Z2}
Consider a behavior $\B$ of the form \eqref{behavior} with minimal state-space representation $(A,B,C,D)$. Let $w\in\B$ and let $x:\Zp \to \mathbb{R}^n$ be a state trajectory of $(A,B,C,D)$ corresponding to $w$. There exist matrices ${Z_1 \in \R^{n\times q\ell}}$ and ${Z_2 \in \R^{n\times q\ell}}$ 
 such that
\begin{align}
\label{x1}
x(t) &= Z_1 w_{[t,t+\ell-1]}, \text{ and} \\
\label{x2}
x(t+\ell) &= Z_2 w_{[t,t+\ell-1]},
\end{align}
for all $t \in \Zp$.
\end{lemma}

\begin{proof}
By the laws of the system, 
$$
y_{[t,t+\ell-1]} = \mathcal{O}_\ell x(t) + \mathcal{T}_\ell
u_{[t,t+\ell-1]},
$$
where we recall that $\mathcal{O}_\ell$ and $\mathcal{T}_\ell$ are defined in Section~\ref{subsec:state-space systems}. Since $(A,B,C,D)$ is a minimal state-space representation, the pair $(C,A)$ is observable and $\mathcal{O}_\ell$ has full column rank. It thus admits a left inverse $\mathcal{O}_\ell^\dagger$ such that 
$$
x(t) = \begin{bmatrix}
-\mathcal{O}_\ell^\dagger \mathcal{T}_\ell & \mathcal{O}_\ell^\dagger
\end{bmatrix} \begin{bmatrix}
u_{[t,t+\ell-1]} \\ y_{[t,t+\ell-1]}
\end{bmatrix}.
$$
By noting that $w_{[t,t+\ell-1]}$ and $\begin{bmatrix}
u_{[t,t+\ell-1]} \\ y_{[t,t+\ell-1]}
\end{bmatrix}$ differ merely by a permutation of rows, we see that there exists a matrix $Z_1$ such that \eqref{x1} holds.

Next, to prove \eqref{x2}, note that $x(t+\ell) = A^\ell x(t) + \mathcal{C}_\ell u_{[t,t+\ell-1]}$. Now, using the fact that 
$$
x(t) = \begin{bmatrix}
-\mathcal{O}_\ell^\dagger \mathcal{T}_\ell & \mathcal{O}_\ell^\dagger
\end{bmatrix} \begin{bmatrix}
u_{[t,t+\ell-1]} \\ y_{[t,t+\ell-1]}
\end{bmatrix},
$$
we obtain 
$$
x(t+\ell) = \begin{bmatrix}
\mathcal{C}_\ell-A^\ell \mathcal{O}_\ell^\dagger \mathcal{T}_\ell & A^\ell \mathcal{O}_\ell^\dagger
\end{bmatrix} \begin{bmatrix}
u_{[t,t+\ell-1]} \\ y_{[t,t+\ell-1]}
\end{bmatrix}.
$$
We conclude that there exists a matrix $Z_2$ satisfying \eqref{x2}. This proves the lemma.
\end{proof}

\section{Proof of Theorem~\ref{t:degreestorage}}
\label{sec:proofdegreestorage}

\begin{proof}[Proof of Theorem~\ref{t:degreestorage}]
Let $Q_\Psi$ be a storage function for $(\B,Q_\Phi)$. By Lemma~\ref{l:psdcoefficient}, we assume, without loss of generality, that $\Psi \geq 0$. If $\deg(Q_\Psi) < \max\{\deg(Q_\Phi),\ell\}$ there is nothing to prove. Suppose that $\deg(Q_\Psi) =: d \geq \max\{\deg(Q_\Phi),\ell\} \geq 0$. We will prove that there exists a storage function of degree $d-1$. Repeated application of this argument thus proves the theorem. 

Let $(A,B,C,D)$ be a minimal state-space representation of $\B$. Let $w \in \B$ be a trajectory of the system for which $w(0),\dots,w(d) = 0$ and, consequently, 
$$
w(d+1) = \begin{bmatrix}
I \\ D
\end{bmatrix} u(d+1),
$$
where $u(d+1) \in \mathbb{R}^m$. This implies that $Q_\Psi(w)(0) = 0$ and 
$$
Q_\Psi(w)(1) = u(d+1)^\top \begin{bmatrix}
I \\ D
\end{bmatrix}^\top \Psi_{d,d} \begin{bmatrix}
I \\ D
\end{bmatrix} u(d+1).
$$
Also, since $\deg(Q_\Phi) \leq d$ we have that $Q_\Phi(w)(0) = 0$. Hence, it follows from the dissipation inequality~\eqref{eq:dissipation_nablapsi} that $Q_\Psi(w)(1) \leq 0$ and thus $Q_\Psi(w)(1) = 0$. Since $\Psi \geq 0$ also $\Psi_{d,d} \geq 0$. We conclude that $\Psi_{d,d} \begin{bmatrix}
I \\ D
\end{bmatrix} u(d+1) = 0$. As $u(d+1)$ is arbitrary, we conclude that 
$$\Psi_{d,d} \begin{bmatrix}
I \\ D
\end{bmatrix} = 0.
$$
Since $\Psi \geq 0$, it holds that $\ker \Psi_{d,d} \subseteq \ker \Psi_{i,d}$ for all $i = 0,\dots,d$. It follows that 
$$
\Psi_{i,d} \begin{bmatrix}
I \\ D
\end{bmatrix} = 0 
$$
for all $i = 0,\dots,d$. 

Now consider an arbitrary trajectory $w \in \B$ with corresponding state trajectory $x:\Zp \to \mathbb{R}^n$. We can write 
$$
w(t+d) = \begin{bmatrix}
I \\ D
\end{bmatrix} u(t+d) + \begin{bmatrix}
0 \\ C
\end{bmatrix}x(t+d) = \begin{bmatrix}
I \\ D
\end{bmatrix} u(t+d) + \begin{bmatrix}
0 \\ C
\end{bmatrix} Z_2 w_{[t+d-\ell,t+d-1]},
$$
where the second equality follows from Lemma~\ref{l:Z1Z2} in \ref{sec:lemstate} and the fact that $d \geq \ell$. This means that 
$$
\Psi_{i,d} w(t+d) = \Psi_{i,d} \begin{bmatrix}
0 \\ C
\end{bmatrix} Z_2 w_{[t+d-\ell,t+d-1]}
$$
for all $i = 0,\dots,d$ and, by symmetry of $\Psi$, also
$$
w(t+d)^\top \Psi_{d,i} = w_{[t+d-\ell,t+d-1]}^\top Z_2^\top \begin{bmatrix}
0 \\ C
\end{bmatrix}^\top
$$
for all $i = 0,\dots,d$. By definition, for any $w \in \B$, 
\begin{align*}
Q_\Psi(w)(t) &= \sum_{i,j=0}^d w(t+i)^\top \Psi_{i,j} w(t+j).
\end{align*}
Using the fact that $d \geq \ell$, we conclude that
$$
Q_\Psi(w)(t) = \begin{cases}
    0 & \text{if } d = 0 \\
    \sum_{i,j=0}^{d-1} w(t+i)^\top \Psi'_{i,j} w(t+j) & \text{if } d \geq 1,
\end{cases}
$$
for some suitable coefficient matrices $\Psi'_{i,j}$, where $i,j = 0,\dots,d-1$. We have thus found a QDF of degree at most $d-1$ that is $\B$-equivalent to $Q_\Psi$. Repetition of this argument establishes the theorem.  
\end{proof}

\section{Proof of Theorem~\ref{theoremequivdiss}}
\label{sec:proofstatestorage}

\begin{proof}[Proof of Theorem~\ref{theoremequivdiss}]
The proof follows immediately in the case that $n = 0$. Therefore, we focus on the case that $n \geq 1$, which implies $\ell \geq 1$ due to observability of $(C,A)$. To prove the `only if' statement, assume that \eqref{sys} is dissipative with respect to $s(u,y)$ and let $x^\top P x$ be a storage function. By Lemma~\ref{l:Z1Z2} in~\ref{sec:lemstate}, we may write $x(t)^\top P x(t) = w_{[t,t+\ell-1]}^\top Z_1^\top P Z_1 w_{[t,t+\ell-1]}$. This shows that the matrix $\Psi := Z_1^\top P Z_1 \geq 0$ defines a storage function $Q_\Psi$ for $(\B,Q_S)$. 

To prove the `if' statement, suppose that $\B$ is dissipative with respect to $Q_S$. By Lemma~\ref{l:psdcoefficient}, there exists a $\Psi \in \mathbb{S}^{qd}$ with $\Psi \geq 0$, such that $Q_\Psi$ is a storage function for $(\B,Q_S)$. The proof follows immediately for the case that $d=0$. Therefore, we assume $d\geq 1$. Let $w\in \B$ be an arbitrary trajectory and  note that 
\begin{equation}
\label{storagefuncproof}
    \begin{aligned}
Q_\Psi(w)(t) &= \begin{bmatrix}
    u_{[t,t+d-1]} \\ y_{[t,t+d-1]}
\end{bmatrix}^\top \Pi^\top \Psi \Pi \begin{bmatrix}
    u_{[t,t+d-1]} \\ y_{[t,t+d-1]}
\end{bmatrix} \\
&= \begin{bmatrix}
    x(t) \\ u_{[t,t+d-1]}
\end{bmatrix}^\top \begin{bmatrix}
    0 & I \\ \mathcal{O}_{ d} & \mathcal{T}_{ d}
\end{bmatrix}^\top \Pi^\top \Psi \Pi \begin{bmatrix}
    0 & I \\ \mathcal{O}_{ d} & \mathcal{T}_{ d}
\end{bmatrix} \begin{bmatrix}
    x(t) \\ u_{[t,t+d-1]} 
\end{bmatrix}.
    \end{aligned}
\end{equation}

Now consider a trajectory $w \in \B$ whose corresponding state $x$ satisfies $x(0) = 0$ and, moreover, $u_{[0,d-1]} =0$. Clearly, $Q_\Psi(w)(0) = 0$ and, since $y(0) = 0$, also $Q_S(w)(0) = 0$. Since $\B$ is dissipative with respect to $Q_S$, we have $Q_\Psi(w)(1) - Q_\Psi(w)(0) \leq Q_S(w)(0)$ and thus $Q_\Psi(w)(1) \leq 0$. Since $\Psi \geq 0$, $Q_\Psi(w)(1)=0$ and, hence,
$$
\Psi \Pi \begin{bmatrix}
    0 & I \\ \mathcal{O}_{ d } & \mathcal{T}_{ d }
\end{bmatrix} \begin{bmatrix}
    x(1) \\ u_{[1,d]}
\end{bmatrix} = 0.
$$
\noindent
Since $w \in \B$ is such that the corresponding state $x$ satisfies $x(0) = 0$ and $u_{[0,d-1]}=0$, we have $x(1) = 0$ and
$$
\Psi \Pi \begin{bmatrix}
    I \\ \mathcal{T}_{ d }
\end{bmatrix}
\begin{bmatrix}
    0 \\ \vdots \\ 0 \\ u({ d }) 
\end{bmatrix} = 0
$$
for all $u({ d }) \in \R^m$. We claim that 
\begin{equation}
    \label{claim}
\Psi \Pi \begin{bmatrix}
    I \\ \mathcal{T}_\ell
\end{bmatrix}
\begin{bmatrix}
    u(1) \\ \vdots \\ u({ d }-1) \\ u({ d }) 
\end{bmatrix} = 0
\end{equation}
for all $u(1),\dots,u({ d }) \in \R^m$. We will prove this claim by induction. If ${ d } = 1$ the claim is true by the above discussion. Hence, suppose that ${ d } > 1$. Assume that for a given $1 < k \leq { d }$,
$$
\Psi \Pi \begin{bmatrix}
    I \\ \mathcal{T}_{ d }
\end{bmatrix}
\begin{bmatrix}
    0 \\ \vdots \\0 \\ u(k) \\ \vdots \\ u({ d }) 
\end{bmatrix} = 0,
$$
for all $u(k),\dots,u({ d }) \in \R^m$. We want to prove that also 
\begin{equation}
    \label{toprove}
\Psi \Pi \begin{bmatrix}
    I \\ \mathcal{T}_{ d }
\end{bmatrix}
\begin{bmatrix}
    0 \\ \vdots \\0 \\ u(k-1) \\ \vdots \\ u({ d }) 
\end{bmatrix} = 0,
\end{equation}
for all $u(k-1),\dots,u({ d }) \in \R^m$. Let $w \in \B$ be a trajectory whose corresponding state $x$ satisfies $x(0) = 0$ and whose input component satisfies $u_{[0,k-2]} = 0$. By the induction hypothesis, $Q_\Psi(w)(0) = 0$. Moreover, $Q_S(w)(0) = 0$ because $y(0) = 0$. Therefore, $Q_\Psi(w)(1) \leq 0$ and thus $Q_\Psi(w)(1) = 0$ by positive semidefiniteness of $\Psi$. Note that $x(1) = 0$. We therefore conclude that \eqref{toprove} holds. By induction, \eqref{claim} holds for all $u(1),\dots,u({ d }) \in \R^m$. Equivalently, we have 
$$
\Psi \Pi \begin{bmatrix}
    I \\ \mathcal{T}_{ d }
\end{bmatrix} = 0.
$$
By \eqref{storagefuncproof}, we obtain 
$$
Q_\Psi(w)(t) = 
    x(t)^\top \begin{bmatrix}
    0 \\ \mathcal{O}_{ d }
\end{bmatrix}^\top \Pi^\top \Psi \Pi \begin{bmatrix}
    0 \\ \mathcal{O}_{ d }
\end{bmatrix} 
    x(t).
$$
We conclude that $V(x) = x^\top P x$ is a storage function for \eqref{sys}, with 
$$
P :=  \begin{bmatrix}
    0 \\ \mathcal{O}_{ d }
\end{bmatrix}^\top \Pi^\top \Psi \Pi \begin{bmatrix}
    0 \\ \mathcal{O}_{ d }
\end{bmatrix}  \geq 0. 
$$
Therefore, \eqref{sys} is dissipative with respect to the supply rate $s(u,y)$. This proves the theorem.  
\end{proof}

\section{Controllability of a pair of matrices depending on a scalar variable} \label{sec:lemcont}

In what follows, we present an auxiliary lemma on the controllability of a pair of matrices depending on a scalar variable $\alpha$. 

\begin{lemma}
\label{lemmacontrollability}
    Consider $A \in \mathbb{R}^{n \times n}$ and $B \in \mathbb{R}^{n \times m}$ and assume that $(A,B)$ is controllable. Let $E(\alpha)$ and $F(\alpha)$ be real polynomial matrices in the indeterminate $\alpha  \in \R$, of sizes $n \times n$ and $n \times m$, respectively. Assume that $E(0) = 0$ and $F(0) = 0$. There exists an $\bar{\alpha} >0$ such that $(A+E(\alpha),B+F(\alpha))$ is controllable for all $\alpha \in \mathbb{R}$ satisfying $|\alpha| \geq \bar{\alpha}$.
\end{lemma}
\begin{proof}
  Define $\hat{A}_\alpha := A+E(\alpha)$ and $\hat{B}_\alpha := B+F(\alpha)$. We have that 
$$
\det \left( \sum_{i = 0}^{n-1} \hat{A}_\alpha^i \hat{B}_\alpha \hat{B}_\alpha^\top (\hat{A}_\alpha^i)^\top \right)
$$
is a nonzero polynomial in $\alpha$ since 
$$
\det \left( \sum_{i = 0}^{n-1} \hat{A}_0^i \hat{B}_0 \hat{B}_0^\top (\hat{A}_0^i)^\top \right) = \det \left( \sum_{i = 0}^{n-1} A^i B B^\top (A^i)^\top \right) \neq 0.
$$
As such, $(\hat{A}_\alpha, \hat{B}_\alpha)$ is uncontrollable only for a finite number of values of $\alpha$. Therefore, there exists an $\bar{\alpha} >0$ such that $(A+E(\alpha),B+F(\alpha))$ is controllable for all $|\alpha| \geq \bar{\alpha}$.
\end{proof}

\bibliographystyle{plain}
\bibliography{ifacconf}

\end{document}